\documentclass[12pt,a4paper,reqno]{amsart}
\usepackage[english]{babel}
\usepackage[utf8]{inputenc}
\usepackage{mathtools}
\usepackage{amsmath}
\usepackage[hidelinks]{hyperref}
\usepackage{amssymb}
\usepackage{ifthen}
\usepackage{graphicx}
\usepackage[capitalise,noabbrev]{cleveref}
\usepackage{pgf}
\usepackage{enumerate}
\usepackage{float}
\usepackage{xpatch}
\usepackage{tikz-cd}
\usepackage[margin=2.5cm]{geometry}
\title{Gorenstein cotorsion pairs induced by balanced pairs}
\author{Konstantinos Golfis}
\date{}

\newcommand{\Z}{\mathbb{Z}}

\newcommand{\A}{\mathcal{A}}
\newcommand{\X}{\mathcal{X}}
\newcommand{\Y}{\mathcal{Y}}

\newtheorem{prop}{Proposition}[section]
\newtheorem{lem}{Lemma}[section]
\newtheorem{defin}{Definition}[section]
\newtheorem{remark}{Remark}[section]
\newtheorem{theor}{Theorem}[section]

\newtheorem{corol}{Corollary}[section]
\newtheorem{ex}{Example}[section]
\def\quotient#1#2{
	\raise1ex\hbox{$#1$}\Big/\lower1ex\hbox{$#2$}
}

\begin{document}

\begin{abstract}
Let $(\X,\Y)$ be an admissible balanced pair in a locally presentable
abelian category $\A$, and let $\mathcal E$ be its induced exact structure.
We study the class $G(\X)$ of cycles of two-sided complexes with terms in
$\X$ that are both right and left $\X$-acyclic. If $\X$ is closed under
direct summands and finite direct sums and is $\kappa$-Kaplansky for every
sufficiently large regular cardinal $\kappa$, we prove that
$(G(\X),G(\X)^{\perp *})$ is a hereditary cotorsion pair relative to
$\mathcal E$. For relative purity determined by finitely presented
modules, we give two criteria for functorial completeness and show that
relative periodic modules generate a functorially complete hereditary
cotorsion pair. We also compute the relative Gorenstein objects, their
right orthogonal, and explicit approximations for a family of exact
structures over Dedekind domains.
\end{abstract}

\maketitle

\section{Introduction}

The classical notion of Gorenstein projective modules dates back to Auslander and Bridger's work in 1969 on modules of G-dimension zero \cite{AuslanderBridger1969}. This concept was later extended to arbitrary modules by Enochs and Jenda in 1995 \cite{EnochsJenda1995} who formalized the concepts of Gorenstein projective and injective modules. Since then, these modules have become a central object of study, with further generalizations introduced by Sather-Wagstaff, Sharif and White \cite{https://doi.org/10.1112/jlms/jdm124}, in order to prove stability and closure properties.

To extend these concepts beyond standard projective and injective objects, it is natural to employ techniques from relative homological algebra. Balanced pairs, introduced by Chen in \cite{CHEN20102718}, provide a framework for this purpose. Let $(\mathcal{X}, \mathcal{Y})$ be a balanced pair in an abelian category $\mathcal{A}$. If the pair is admissible, it naturally endows $\mathcal{A}$ with an exact category structure in the sense of Quillen \cite{Quillen1973}. In this setting, the classical Yoneda Ext functor is replaced by the relative extension functor, denoted $\operatorname{Ext}_{\mathcal{X}}$.

A basic question is whether a Gorenstein class and its right orthogonal
form a hereditary cotorsion pair, and whether that pair is complete.
Cort\'{e}s-Izurdiaga and \v{S}aroch \cite{CortesIzurdiaga2025} proved the
cotorsion-pair assertion for ordinary Gorenstein projective modules over
arbitrary rings and obtained sufficient conditions for completeness.
Smallness conditions, including deconstructibility and Kaplansky
conditions, provide useful tools for such questions \cite{Stovicek2013}.
In an efficient exact category, a set of objects whose transfinite
extension closure contains a generator generates a functorially complete
cotorsion pair \cite[Theorem~5.16]{StovicekExactModels2014}.
Set-theoretic approximation methods also play an important role in the
proof of the flat cover conjecture \cite{Bicanetal2001}.

In Section~2, we recall the background on relative homological algebra,
balanced pairs, and exact categories. In Section~3, we study $G(\X)$,
the class of cycles of complexes with terms in $\X$ that are
both right and left $\X$-acyclic. Adapting the arguments of
\cite{CortesIzurdiaga2025}, we prove that
$(G(\X),G(\X)^{\perp *})$ is a hereditary $*$-cotorsion pair when
$\X$ is closed under direct summands and finite direct sums and is
$\kappa$-Kaplansky for every sufficiently large regular cardinal
$\kappa$. The latter smallness condition holds, in particular, for
deconstructible classes in Grothendieck categories.

Section~4 treats relative purity determined by a set of finitely
presented modules. We first verify the Kaplansky condition and apply
Section~3. We then adapt the periodic filtration method of
\cite{CortesIzurdiaga2025} to obtain functorial completeness under either
uniform $\lambda$-pure-injectivity of the relative projectives or an
orthogonality condition on relative periodic modules. Independently of
these hypotheses, periodic modules generate a functorially complete
hereditary relative cotorsion pair, and periodic orthogonality can be
tested on a set of modules of bounded cardinality. Finally, for tests
$R/\mathfrak p^{n_{\mathfrak p}}$ over a Dedekind domain, we compute
both classes of the Gorenstein cotorsion pair and construct explicit
approximations, allowing infinitely many distinct primes.

\section{Preliminaries}
\subsection{Relative Homological Algebra}
We recall some necessary notions of relative homological algebra, following \cite{EnochsJenda2000}. Let $\A$ be an abelian category and $\X$ be a class of objects of $\A$. A morphism $f:B\to A$ between two objects $A, \ B \in \A$ is called $\X$-epic if the morphism $\operatorname{Hom}_\A(X,f)$ is an epimorphism for all $X\in \X$. An $\X$-precover of $A$ is an $\X$-epic morphism $f:X\to A$ with $X\in \X$. The class $\X$ is called precovering if every object $A\in \A$ has an $\X$-precover. Dually, for a class $\Y$, we have the notions of $\Y$-monic, $\Y$-preenvelope and  preenveloping class. Moreover, the class $\X$ is called admissible if every $\X$-precover is also an epimorphism and dually the class $\Y$ is called coadmissible if every $\Y$-preenvelope is also a monomorphism.
\begin{defin}
    Let $C_\bullet$ be a complex in $Ch(\mathcal{A})$. We say that $C_\bullet$ is right (resp. left) $\X$-acyclic if the complex $\operatorname{Hom}_\mathcal{A}(X,C_\bullet)$ (resp. $\operatorname{Hom}_\mathcal{A}(C_\bullet,X)$) is acyclic for every $X \in \X$. 
\end{defin}
\noindent A (left) $\X$-resolution of $A$ is a right $\X$-acyclic complex
     \begin{equation*}
        \dots \to X_1 \to X_0 \to A \to 0 
    \end{equation*}
with each $X_i \in \X$. If $\X$ is precovering, then every object $A\in \A$ has an $\X$-resolution by \cite[Proposition 8.1.3]{EnochsJenda2000}, which is unique up to homotopy by standard arguments. These resolutions can be used to define derived functors relative to $\X$. If $T:\A^{\mathrm{op}}\to\mathbf{Ab}$ is an additive functor and $X_\bullet$ a deleted $\X$-resolution, then the right derived functors $R^nT$ are the nth cohomology groups of $T(X_\bullet)$. Dually, if $\Y$ is preenveloping, we have the notions of a (right) $\Y$-coresolution and for an additive functor $S:\A\to\mathbf{Ab}$ we have the right derived functors relative to $\Y$. These functors have expected properties, such as a long exact sequence, but we have to replace short exact sequences by right $\X$-acyclic ones or left $\Y$-acyclic ones, respectively, as seen in the following theorem.
\begin{theor}\cite[Theorem 8.2.3]{EnochsJenda2000}\label{enochsjenda} Let $\X$ be a precovering class closed under finite direct sums in an abelian category $\mathcal{A}$ and $0 \to M' \to M \to M'' \to 0$ be a right $\X$-acyclic complex of objects of $\mathcal{A}$. Then, for any additive functor $T:\A^{\mathrm{op}}\to\mathbf{Ab}$, there is a long exact sequence
    \begin{align*}
        0 &\to R^0 T(M'') \to R^0 T(M) \to R^0 T(M') \to \dots \\
        &\to R^{n-1} T(M'') \to R^{n-1} T(M) \to R^{n-1} T(M') \to R^n T(M'') \to \dots.
    \end{align*}

\end{theor}
For $M,N\in\A$, apply these constructions to
$T=\operatorname{Hom}_{\A}(-,N)$ and
$S=\operatorname{Hom}_{\A}(M,-)$. We write
\[
\operatorname{Ext}^n_{\X}(M,N)=R^nT(M),\qquad
\operatorname{Ext}^n_{\Y}(M,N)=R^nS(N).
\]
\begin{remark} 
    Let $f:M\to N$ be a morphism and $X_1\to X_0\to M\to 0$ be the beginning of an $\X$-resolution of $M$. Then the following diagram commutes
    \[\begin{tikzcd}
	{X_1} & {X_0} & M & 0 \\
	&& N
	\arrow["{d_1}", from=1-1, to=1-2]
	\arrow["0"', from=1-1, to=2-3]
	\arrow["{\varepsilon }", from=1-2, to=1-3]
	\arrow["{f\varepsilon }", from=1-2, to=2-3]
	\arrow[from=1-3, to=1-4]
	\arrow["f", from=1-3, to=2-3]
\end{tikzcd}\]
Thus, we have a map $\operatorname{Hom}_\A(M,N)\to\operatorname{Ext}^0_\X(M,N)=\ker(d_1^*)$. Additionally, if the complex $X_1\to X_0\to M\to 0$ is acyclic (for example, if $\X$ is admissible), then this map is an isomorphism by the universal property of the cokernel. Dual results hold for $\operatorname{Ext}^0_\Y(M,N)$.
\end{remark}
In order to justify the name of these functors, we note that if $\X$ is admissible, then every right $\X$-acyclic complex is, in fact, acyclic. Therefore $\operatorname{Ext}^1_\X(M,N)$ is a subfunctor of the classical Yoneda $\operatorname{Ext}^1_\A(M,N)$, consisting of isomorphism classes of right $\X$-acyclic short exact sequences $0\to N\to K\to M\to 0$, see \cite{Auslander01011993}.

\subsection{Balanced Pairs}
Balanced pairs were introduced by Chen in \cite{CHEN20102718} and provide an ideal framework for doing relative homological algebra. Enochs and Jenda in \cite{EnochsJenda2000} defined the notion of right balancedness for a functor $T$. These two notions coincide in the case where $T=\operatorname{Hom}_\A$.
The following equivalent formulation is
\cite[Proposition~2.2]{CHEN20102718}.
\begin{defin}
    Let $\A$ be an abelian category. A pair $(\X,\Y)$ of subclasses of objects of $\A$ is called a balanced pair if $\X$ is precovering, $\Y$ is preenveloping and the class of right $\X$-acyclic complexes coincides with the class of left $\Y$-acyclic complexes.
\end{defin}
 By \cite[Corollary 2.3]{CHEN20102718} if $(\X,\Y)$ is a balanced pair, $\X$ is admissible if and only if $\Y$ is coadmissible. In this case, the balanced pair is called admissible. The following are standard examples of balanced pairs.
\begin{ex}
Let $(\X,\mathcal Z,\Y)$ be a complete hereditary cotorsion triple,
that is, both $(\X,\mathcal Z)$ and $(\mathcal Z,\Y)$ are complete
hereditary cotorsion pairs. Then $(\X,\Y)$ is an admissible balanced
pair by \cite[Proposition~4.2]{estrada2020balanced}; see also
\cite[Proposition~2.6]{CHEN20102718}.
An abelian category admits a complete hereditary cotorsion triple if
and only if it has enough projectives and injectives
\cite[Theorem~4.4]{estrada2020balanced}. In that case,
$(\operatorname{Proj},\operatorname{All},\operatorname{Inj})$ gives
the standard balanced pair $(\operatorname{Proj},\operatorname{Inj})$.
\end{ex}
\begin{ex}
    By \cite[Example 8.3.2]{EnochsJenda2000} the pair $(PProj,PInj)$ of pure projective and pure injective objects is a balanced pair in the setting of module categories. More generally, the same construction applies in a locally finitely presented abelian category: the pure exact structure has enough pure projectives and pure injectives; see \cite{CrawleyBoevey1994}.
\end{ex}
\begin{ex}
In the category of left $R$-modules, denote by $\operatorname{GProj}$ and
$\operatorname{GInj}$ the ordinary Gorenstein projective and
Gorenstein injective classes. Suppose that $R$ is left virtually
Gorenstein, meaning that
$\operatorname{GProj}^{\perp}={}^{\perp}\operatorname{GInj}$.
Then
$(\operatorname{GProj},\operatorname{GProj}^{\perp},
\operatorname{GInj})$ is a complete hereditary cotorsion triple,
and $(\operatorname{GProj},\operatorname{GInj})$ is an admissible
balanced pair.
\end{ex}
Balancedness gives natural isomorphisms
\[
\operatorname{Ext}^n_{\X}(M,N)\cong
\operatorname{Ext}^n_{\Y}(M,N)
\qquad(n\geq0),
\]
for all $M,N\in\A$; see \cite[Lemma~2.1]{CHEN20102718}.
We denote their common value by $\operatorname{Ext}^n_*(M,N)$.
This allows dimension shifting in both variables.
\subsection{Exact Categories and Cotorsion Pairs}
Exact categories were introduced by Quillen in \cite{Quillen1973} and an axiomatic description can be found in \cite{Buhler2010}. An exact category $(\A,\mathcal{E})$ is an additive category $\A$ equipped with a class $\mathcal{E}$ of kernel-cokernel pairs $A\overset{i}{\to}{B}\overset{p}{\to}C$, called conflations, which satisfy certain axioms, resembling those of short exact sequences. The maps $i$ are called inflations and the maps $p$ are called deflations. This class $\mathcal{E}$ permits the definition of a Yoneda Ext functor, denoted $\operatorname{Ext}^1_{\mathcal E}$, see \cite[Chapter 1]{gillespie2025abelian} for details.
\begin{defin}
A pair of classes $(\mathcal F,\mathcal C)$ in $(\A,\mathcal E)$ is
a $*$-cotorsion pair if $\mathcal F={}^{\perp_*}\mathcal C$ and
$\mathcal C=\mathcal F^{\perp_*}$, where
\[
{}^{\perp_*}\mathcal C
=\{F:\operatorname{Ext}^1_{\mathcal E}(F,C)=0
\text{ for all }C\in\mathcal C\},
\]
and $\mathcal F^{\perp_*}$ is defined dually.
It is hereditary if $\mathcal F$ is closed under kernels of deflations
between its objects and $\mathcal C$ is closed under cokernels of
inflations between its objects.
It is complete if every $M\in\A$ admits conflations
\[
0\to C\to F\to M\to0,
\qquad
0\to M\to C'\to F'\to0,
\]
with $F,F'\in\mathcal F$ and $C,C'\in\mathcal C$.
The maps $F\to M$ and $M\to C'$ are called a special
$\mathcal F$-precover and a special $\mathcal C$-preenvelope,
respectively. The pair is functorially complete if these two
approximation conflations can be chosen functorially in $M$.
\end{defin}
Let $\A$ be an abelian category and  $(\X,\Y)$ an admissible balanced pair. Then the class $\mathcal{E}$ of all right $\X$-acyclic short exact sequences
makes $(\A,\mathcal{E})$ an exact category with $\operatorname{Ext}^1_{\mathcal E}=\operatorname{Ext}^1_*$. See \cite[Lemma 2.3]{Lietal2015} for the proof of closure axioms required for the exact structure. We also assume that $\X$ is closed under direct summands and finite direct sums. Finite sums of objects of $\Y$ are relative injective. Thus,
if needed, replacing $\Y$ by its closure under finite direct sums
does not change the induced exact structure or balancedness.
We use this convention throughout. This is our framework for the
rest of the paper.

\section{Relative Gorenstein objects}
 Following the generalizations by Sather-Wagstaff, Sharif and White referenced in the introduction \cite{https://doi.org/10.1112/jlms/jdm124} we define our class of interest, the Gorenstein $\X$-objects. Instead of using standard projective objects, we define Gorenstein objects relative to a general class $\X$ by replacing standard acyclicity with right and left $\X$-acyclicity.
\begin{defin}
    Let $\X$ be a class of objects of $\A$. We say that an object $M \in \A$ is a Gorenstein $\X$-object if there exists a complex that is both right and left $\X$-acyclic, $X_\bullet=\dots\to X_{1}\to X_0\to X_{-1}\to\dots$ such that each $X_i \in \X$ for all $i \in \Z$ and $M= Z_0(X_\bullet)$. Denote by $G(\X)$ the class of all Gorenstein $\X$-objects.
\end{defin}

   Observe that $Z_i(X_\bullet) \in G(\X)$ for all $i \in \Z$. The augmented sequence
\[
\cdots\to X_{i+2}\to X_{i+1}\to Z_i(X_\bullet)\to0
\]
is a right $\X$-acyclic $\X$-resolution of $Z_i(X_\bullet)$.

\begin{remark}
    We have the inclusion $\X\subseteq G(\X)$. Indeed, let $X \in \X$. The complex with $X\xrightarrow{1_X}X$ in homological degrees $1$ and $0$, and zero elsewhere, is contractible and has zeroth cycle $X$. It is therefore both right and left $\X$-acyclic. Additionally, we have the inclusion $\X\subseteq G(\X)^{\perp_*}$. Indeed, let $M\in G(\X)$ and $X\in \X$. Then $M$ has a right and left $\X$-acyclic $\X$-resolution and this can be used to compute the relative $\operatorname{Ext}_*$ functor. Hence $\operatorname{Ext}_*^n(M,X)=0$ for all $n\geq 1$.
\end{remark}

   This definition is an immediate generalisation of the classical case, for if $\X$ is the class of projective (resp. injective) modules  in $R$-Mod, then $G(\X)$ is the class of Gorenstein projective (resp. injective) modules. Below is a different example.

\begin{ex}\label{pureprojectivepureinjective}
    Let $\X=PProj$ be the class of pure projective modules in $R$-Mod. Then $G(PProj)=PProj$. Indeed, if a complex $X$ is right $PProj$-acyclic then it is pure acyclic by \cite[Theorem 3.6]{emmanouil2016pure}. But a pure acyclic complex of pure projective objects is contractible by \cite[Corollary 3.7]{emmanouil2016pure}, therefore every cycle $Z_n(X)$ is a direct summand of a pure projective module and thus pure projective itself. Since $PProj\subseteq G(PProj)$ we have the desired equality. 
\end{ex} 
\subsection{Heredity}
We first establish the vanishing and closure properties needed for
heredity. These arguments use the relative long exact sequences.

\begin{lem}\label{exti}
    Let $M \in G(\X)$, $N\in G(\X)^{\perp*}$. Then $\operatorname{Ext}^i_*(M,N)=0$ for all $i\geq 1$.
\end{lem} 
\begin{proof}
Let $X_\bullet$ be a right and left $\X$-acyclic complex such that $Z_0(X_\bullet)=M$.
For each $j$, its cycle sequence
\[
0\to Z_j(X_\bullet)\to X_j\to Z_{j-1}(X_\bullet)\to0
\]
is a relative conflation. Since each $X_j$ is relatively projective,
successive dimension shifting gives
\[
\operatorname{Ext}^i_*(M,N)
\cong\operatorname{Ext}^1_*(Z_{i-1}(X_\bullet),N)=0
\qquad(i\geq1).
\]
The last equality follows from $Z_{i-1}(X_\bullet)\in G(\X)$
and $N\in G(\X)^{\perp *}$.
\end{proof} 
\begin{lem}
    $G(\X)^{\perp*}$ has the two-out-of-three property for relative conflations.
\end{lem}
\begin{proof}
Consider a relative conflation
\begin{equation}\label{eq1}
0\to N'\to N\to N''\to0.
\end{equation}
If $N',N''\in G(\X)^{\perp *}$, the relative long exact sequence gives
$\operatorname{Ext}^1_*(M,N)=0$ for every $M\in G(\X)$, so
$N\in G(\X)^{\perp *}$. If $N',N\in G(\X)^{\perp *}$, the same sequence and
\cref{exti} give
\[
\operatorname{Ext}^1_*(M,N'')\hookrightarrow
\operatorname{Ext}^2_*(M,N')=0,
\]
hence $N''\in G(\X)^{\perp *}$.

\noindent Finally, suppose $N,N''\in G(\X)^{\perp *}$. For $M\in G(\X)$, choose
a conflation
\begin{equation}\label{eq2}
0\to M\to X\to M'\to0,
\qquad X\in\X,\quad M'\in G(\X).
\end{equation}
The long exact sequence for \eqref{eq1} gives
\[
\operatorname{Ext}^1_*(M',N'')\to
\operatorname{Ext}^2_*(M',N')\to
\operatorname{Ext}^2_*(M',N).
\]
Its outside terms vanish by \cref{exti}. Relative projectivity of
$X$ and \eqref{eq2} give
$\operatorname{Ext}^1_*(M,N')\cong
\operatorname{Ext}^2_*(M',N')=0$. Thus $N'\in G(\X)^{\perp *}$.
\end{proof}
\begin{remark}\label{applicationsofbalancedpairs}
By \cite[Theorem 3.8]{Lietal2015} combined with the above lemma, we get that the pair $({}^{\perp *}(G(\X)^{\perp *}),G(\X)^{\perp*})$ is a hereditary $*$-cotorsion pair. This means that ${}^{\perp *}(G(\X)^{\perp *})$ is closed under kernels of $\X$-epic morphisms between its objects.
    
\end{remark}
\subsection{Kaplansky condition}
For the remainder of the section, we will further assume that $\A$ is locally presentable. We need this condition to guarantee the existence of set-indexed products, as well as certain presentability properties of objects. Crucial properties that we will be using are the following. Firstly, every object is $\kappa$-generated for some cardinal $\kappa$. Secondly, the image of a $<\kappa$-generated object is again $<\kappa$-generated. Lastly, isomorphism classes of $\kappa$-presentable objects form a set, not a proper class. Let us now proceed to the set-theoretic conditions we shall impose on the balanced pair.
\begin{defin}
Let $\kappa$ be a regular cardinal. A class $\X$ is $\kappa$-Kaplansky if for every $F \in \X$ and every $<\kappa$-generated $M \subseteq F$ there exists $M \subseteq N \subseteq F$ such that $N$ is $<\kappa$-presented and $N \in \X$ and $F/N \in \X$.
\end{defin}
The definition of Kaplansky classes seems restrictive, but the following proposition gives us an important family of classes with this property. A class of objects $\X$ is called deconstructible if there exists a set $\mathcal{S}$ such that all objects of $\X$ are filtered by objects of $\mathcal{S}$ so that $\X=\operatorname{Filt}(\mathcal{S})$, see \cite{Stovicek2013} for details.
\begin{prop}\cite[Corollary 2.7]{Stovicek2013}
If $\A$ is a Grothendieck category and $\X$ is a deconstructible class, then there exists a regular cardinal $\lambda $ such that $\X$ is $\kappa $-Kaplansky for every regular cardinal $\kappa  \ge \lambda $. 
\end{prop}
 We can now prove that $(G(\X), G(\X)^{\perp*})$ is a $*$-cotorsion pair. Recently, Cort\'{e}s-Izurdiaga and \v{S}aroch handled the case of the Gorenstein projective cotorsion pair, and we can generalize their arguments here. We will show that ${}^{\perp *}(G(\X)^{\perp *}) = G(\X)$. We begin by generalizing \cite[Lemma 3.1]{CortesIzurdiaga2025} and then proceed with the main theorem.  
\begin{lem}\label{izur}Let $\X$ be a class such that there exists a regular cardinal $\lambda$ such that $\X$ is $\lambda '$-Kaplansky for every regular cardinal $\lambda '\geq \lambda $.
Let $M \in {}^{\perp *}(G(\X)^{\perp *})$ be $\kappa $-generated for some $\kappa $. Let $\kappa '=\max(\lambda ,\kappa )^+$. Then there exists a $\kappa '$-presentable $L \in \X \subseteq G(\X)^{\perp*}$ and $\varphi : M \to L$ such that for every $N \in G(\X)^{\perp*}$ and every $f:M \to N$, there exists $g:L\to N$ such that $f=g\varphi $.
\end{lem}

\begin{proof}
If not, we can pick representatives from the set of isomorphism classes of $\kappa '$-presentable objects. For every such $\kappa'$-presentable representative $L \in \X$ and for every $\varphi: M \to L$, there exist $N_{L,\varphi} \in G(\X)^{\perp*}$ and $f_{L,\varphi}: M \to N_{L,\varphi}$ that does not factor through $\varphi$. Form the set-indexed product
$N=\prod_{L,\varphi}N_{L,\varphi}$.
This product belongs to $G(\X)^{\perp *}$. Indeed, for fixed
$G\in G(\X)$, compute relative Ext using a fixed deleted
$\X$-resolution of $G$. The functor $\operatorname{Hom}_{\A}$
commutes with products in its second variable, and products in
$\mathbf{Ab}$ are exact. Therefore
\[
\operatorname{Ext}^1_*\left(G,\prod_{L,\varphi}N_{L,\varphi}\right)
\cong\prod_{L,\varphi}\operatorname{Ext}^1_*(G,N_{L,\varphi})=0.
\]
Next, choose an $\X$-precover of $N$ :

\[0 \to K \to L \xrightarrow{p} N \to 0\] 

Let $f:M \to N$ be defined by $f_{L,\varphi }$ as its components. The above short exact sequence is right $\X$-acyclic and $L \in \X \subseteq G(\X)^{\perp *}$. Since $G(\X)^{\perp*}$ has the two-out-of-three property for such conflations, we have $K \in G(\X)^{\perp*}$, thus $\operatorname{Ext}^1_*(M,K) = 0$. So, there exists $g:M \to L $ such that $pg=f$.
\[\begin{tikzcd}
	&&& M & \\
	0 & K & L & N & 0
	\arrow["g"', dashed, from=1-4, to=2-3]
	\arrow["f", from=1-4, to=2-4]
	\arrow[from=2-1, to=2-2]
	\arrow[from=2-2, to=2-3]
	\arrow["p", from=2-3, to=2-4]
	\arrow[from=2-4, to=2-5]
\end{tikzcd}\]

Since $M$ is $\kappa $-generated, it is $<\kappa '$-generated, and so is $\operatorname{im}(g)\subseteq L$. Since $\X$ is $\kappa '$-Kaplansky, there exists a $\kappa '$-presentable $L'\subseteq L$ that contains $\operatorname{im}(g)$, with $L' \in \X$ and $L/L' \in \X$. Then $g$ factors as 

\[M \xrightarrow{\gamma } L' \xrightarrow{i} L\]

We may identify $L'$ with its chosen representative, transporting
$i$ and $\gamma$ along that isomorphism. Combining the two diagrams,
we have 
\[\begin{tikzcd}
	&& {L'} & M & \\
	0 & K & L & N & 0 \\
	&&& {N_{L',\gamma }}
	\arrow["i"', from=1-3, to=2-3]
	\arrow["{\gamma }"', from=1-4, to=1-3]
	\arrow["g"', from=1-4, to=2-3]
	\arrow["f", from=1-4, to=2-4]
	\arrow["{f_{L',\gamma }}"{pos=0.2}, bend left=35, from=1-4, to=3-4]
	\arrow[from=2-1, to=2-2]
	\arrow[from=2-2, to=2-3]
	\arrow["p", from=2-3, to=2-4]
	\arrow[from=2-4, to=2-5]
	\arrow["\pi _{L',\gamma }"', from=2-4, to=3-4]
\end{tikzcd}\]
We observe that $f_{L',\gamma }=\pi _{L',\gamma }f=\pi _{L',\gamma }(pg)=\pi _{L',\gamma }p(i\gamma )=(\pi _{L',\gamma }pi)\gamma $, which means that $f_{L',\gamma }$ factors through $\gamma $, a contradiction.
\end{proof}
We note that the above proof works with a milder condition than being a Kaplansky class, as we did not use the fact that $L/L' \in \X$, only that $L'\in \X$. 
\begin{remark} \label{rem1}
If $(\X,\Y)$ is an admissible balanced pair, then the morphism $\varphi :M\to L$ of the previous lemma is, in fact, a monomorphism.
\end{remark}
 Indeed, let $f:M\to Y$ be a $\Y$-preenvelope. Since $\Y$ is coadmissible this is a monomorphism. Additionally, we have $\operatorname{Ext}^1_*(-,Y)=0$, hence $Y \in All^{\perp*} \subseteq G(\X)^{\perp*}$ and therefore there exists $g:L\to Y$ such that $g\varphi =f$. Finally, $\ker(\varphi )\subseteq\ker(f)=0$, so $\varphi $ is a monomorphism.\\
 We can now state and prove the main theorem of this section. We repeat the necessary conditions for the convenience of the reader.

\begin{theor}\label{cotorsion pair}
Let $\A$ be a locally presentable abelian category and $(\X,\Y)$ be an admissible balanced pair, with $\X$ closed under direct summands and finite direct sums. Additionally, assume there exists a regular cardinal $\lambda$ such that $\X$ is $\kappa $-Kaplansky for all regular cardinals $\kappa  \geq \lambda $. 
Then $(G(\X), G(\X)^{\perp *})$ is a hereditary $*$-cotorsion pair.
\end{theor}

\begin{proof}
We only need to show ${}^{\perp *}(G(\X)^{\perp *}) \subseteq G(\X)$.

Let $M \in {}^{\perp *}(G(\X)^{\perp *})$. Let $\varphi :M\to L$ be as in Lemma \ref{izur} and take the cokernel $N$.
\[0 \rightarrow M \xrightarrow{\varphi} L \rightarrow N \rightarrow 0\]

Observe that the above sequence is a short exact sequence by Remark \ref{rem1}. It is also left $\Y$-acyclic. Indeed, let $Y\in \Y$ and $f:M\to Y$. We have $\Y\subseteq All^{\perp*}\subseteq G(\X)^{\perp*}$. The first inclusion follows from the fact that $\operatorname{Ext}_{*}^1(-, Y) = 0$ for every $Y \in \Y$. So there exists $g:L\to Y$ such that $f=g\varphi $.
\[
\begin{tikzcd}
0 \arrow[r] & M \arrow[r, "\varphi"] \arrow[d, "f"'] & L \arrow[r] \arrow[dl, "g", dashed] & N \arrow[r] & 0 \\
 & Y & & & 
\end{tikzcd}
\]
Now, let $T \in G(\X)^{\perp*}$. From the long exact sequence of $\operatorname{Ext}_*$ we have the exact sequence
\[\operatorname{Hom}_\A(L,T)\xrightarrow{\varphi ^*} \operatorname{Hom}_\A(M,T)\xrightarrow{}\operatorname{Ext}^1_*(N,T)\to \operatorname{Ext}^1_*(L,T)\]
But $\varphi ^*$ is an epimorphism from Lemma \ref{izur} and $\operatorname{Ext}^1_*(L,T)=0$, so $\operatorname{Ext}^1_*(N,T)=0$ and hence $N\in {}^{\perp*}(G(\X)^{\perp*})$. Setting $X^0=L$ and repeating the construction gives an augmented
coresolution
\[
0\to M\to X^0\to X^1\to\cdots
\]
that is left $\Y$-acyclic and hence right $\X$-acyclic.
Its deleted complex $X^\bullet_+=(0\to X^0\to X^1\to\cdots)$
has $Z^0(X^\bullet_+)=M$, and every cycle belongs to
${}^{\perp *}(G(\X)^{\perp *})$. For the other half of the resolution, take an $\X$ resolution of $M$
\[\cdots\to X^{-n}\to\cdots \to X^{-1}\to M\to 0\]
This is a right $\X$-acyclic complex with $X^{-n}\in \X$ for all $n>0$. We show inductively that every kernel belongs to ${}^{\perp*}(G(\X)^{\perp*})$. Indeed, for the first step, if $K=\operatorname{im}(X^{-2}\to X^{-1})$, the following is a right $\X$-acyclic complex
\[0\to K\to X^{-1}\to M\to 0\]
From Remark \ref{applicationsofbalancedpairs} we have $K\in {}^{\perp *}(G(\X)^{\perp *})$. 
Gluing the resolution of $M$ to $X^{\bullet}_+$ we get a right $\X$-acyclic complex
\[\begin{tikzcd}
	{X^\bullet :} & \cdots & {X^{-1}} && {X^0} & \cdots \\
	&&& M
	\arrow[from=1-2, to=1-3]
	\arrow[from=1-3, to=1-5]
	\arrow[from=1-3, to=2-4]
	\arrow[from=1-5, to=1-6]
	\arrow[from=2-4, to=1-5]
\end{tikzcd}\]
It remains to show that $X^\bullet$ is left $\X$-acyclic.
For every $j\in\mathbb Z$ there is a relative conflation
\[
0\to Z^j(X^\bullet)\to X^j\to Z^{j+1}(X^\bullet)\to0.
\]
All these cycles belong to ${}^{\perp *}(G(\X)^{\perp *})$.
If $A\in\X\subseteq G(\X)^{\perp *}$, then
$\operatorname{Ext}^1_*(Z^{j+1}(X^\bullet),A)=0$.
Thus every map $Z^j(X^\bullet)\to A$ extends to $X^j$.
Applying this in each degree proves that
$\operatorname{Hom}_{\A}(X^\bullet,A)$ is acyclic.
Consequently $M\in G(\X)$, and heredity follows from
\cref{applicationsofbalancedpairs}.
\end{proof}

\section{Relative purity and completeness}
\label{sec:set-relative-purity}

Section~3 supplies the hereditary cotorsion pair once the Kaplansky
condition is verified. We now turn to completeness for relative purity
determined by finitely presented modules. The main construction expresses
Gorenstein objects as direct summands of objects filtered by a set of
small modules. Periodic modules allow us to carry out this construction
while keeping track of relative exactness.

Throughout this section, $R$ is a ring and $\mathcal S$ is a set of finitely
presented left $R$-modules. Put
\[
\mathcal S_0=\mathcal S\cup\{R\},\qquad
\X_{\mathcal S}=\operatorname{Add}(\mathcal S_0),\qquad
G_0=\bigoplus_{S\in\mathcal S_0}S.
\]
Let $\mathcal E_{\mathcal S}$ consist of the short exact sequences which
remain exact under $\operatorname{Hom}_R(S,-)$ for every $S\in\mathcal S_0$.
Adjoining $R$ does not change this exact structure. All extension groups
$\operatorname{Ext}^i_*$ and orthogonals in this section are relative to
$\mathcal E_{\mathcal S}$. The notation $\operatorname{Filt}_{\mathcal E_{\mathcal S}}(\mathcal C)$
means transfinite extensions with factors in $\mathcal C$ and with all
filtration inclusions $\mathcal E_{\mathcal S}$-inflations.
The notation $\operatorname{Retr}(\mathcal C)$ denotes closure under
direct summands and isomorphisms.

\subsection{The balanced pair and the Kaplansky condition}

\begin{prop}\label{Sbalancedpair}
Let $\Y_{\mathcal S}$ be the class of $\mathcal E_{\mathcal S}$-injective
modules. Then $(\X_{\mathcal S},\Y_{\mathcal S})$ is an admissible balanced
pair inducing $\mathcal E_{\mathcal S}$, and
\[
\operatorname{Proj}(R\text{-}\mathrm{Mod},\mathcal E_{\mathcal S})
=\X_{\mathcal S}.
\]
Moreover, $(R\text{-}\mathrm{Mod},\mathcal E_{\mathcal S})$ is an efficient
exact category with generator $G_0$.
\end{prop}

\begin{proof}
For every module $M$, the evaluation morphism
\[
\bigoplus_{S\in\mathcal S_0}S^{(\operatorname{Hom}_R(S,M))}
\longrightarrow M
\]
is epic and is surjective under each $\operatorname{Hom}_R(S,-)$.
Its domain belongs to $\X_{\mathcal S}$. Every object of
$\X_{\mathcal S}$ is $\mathcal E_{\mathcal S}$-projective, and this
evaluation splits whenever $M$ is $\mathcal E_{\mathcal S}$-projective.
This proves the asserted description of the projectives and also shows
that every $\X_{\mathcal S}$-precover is epic. Relative-purity theory
provides enough $\mathcal S_0$-pure-injectives and balances
$\operatorname{Hom}_R(-,-)$ by the relative projectives and injectives;
see \cite[Proposition~2.8 and Corollary~2.9]{ZarehDivaani2013}.
Thus the pair is admissible and induces the stated exact structure.

We verify efficiency in the sense of
\cite[Definition~3.4]{StovicekExactModels2014}.
The underlying abelian category is weakly idempotent complete.
Filtered colimits preserve $\mathcal E_{\mathcal S}$-conflations:
they are exact in $R\text{-}\mathrm{Mod}$, and
$\operatorname{Hom}_R(S,-)$ preserves them for every finitely presented
$S\in\mathcal S_0$. Consequently, transfinite composites of
$\mathcal E_{\mathcal S}$-inflations exist and are again inflations.
Every module $L$ is small relative to monomorphisms: for regular
$\theta>|L|$, every map from $L$ into a continuous union indexed by an
ordinal of cofinality at least $\theta$ factors through one stage, and
equality of such maps is detected there since all transition maps are
monic. Finally, the evaluation $G_0^{(\operatorname{Hom}_R(G_0,M))}\to M$
is an $\mathcal E_{\mathcal S}$-deflation. Indeed, every map $S\to M$
extends to $G_0$ by zero on the other summands. Thus $G_0$ is a generator
of the exact category.
\end{proof}

\begin{lem}\label{AddSKaplansky}
For every $R$-module $P$, the class $\operatorname{Add}(P)$ is
$\kappa$-Kaplansky whenever $\kappa$ is regular and
\[
\kappa>\max\{|R|,|P|,\aleph_0\}.
\]
If $\mathcal T$ is any set of finitely presented modules, then
$\operatorname{Add}(\mathcal T)$ is $\kappa$-Kaplansky for every
uncountable regular $\kappa$. In both assertions the small intermediate
submodule can be chosen to be a direct summand.
\end{lem}

\begin{proof} Let $X$ belong to either $\operatorname{Add}(P)$ in the first case or $\operatorname{Add}(\mathcal T)$ in the second. Write $X=\operatorname{Im}e$ for an idempotent on
$Q=\bigoplus_{i\in I}P_i$, where all $P_i=P$ in the first assertion and
$P_i\in\mathcal T$ in the second. For each $i$, choose
$\sigma(i)\subseteq I$ such that
\[
e(P_i)\subseteq\bigoplus_{j\in\sigma(i)}P_j.
\]
In the first case $|\sigma(i)|\leq\max\{|P|,\aleph_0\}$; in the
second case $\sigma(i)$ can be finite, since $P_i$ is finitely generated.
Given a submodule $U\subseteq X$ generated by fewer than $\kappa$
elements, let $J_0$ contain their supports and set
\[
J_{n+1}=J_n\cup\bigcup_{i\in J_n}\sigma(i),\qquad
J=\bigcup_{n<\omega}J_n.
\]
The cardinal assumptions give $|J|<\kappa$. Put
$Q_J=\bigoplus_{j\in J}P_j$ and $V=e(Q_J)$.
Then $e(Q_J)\subseteq Q_J$, $U\subseteq V$, and $V$ is a retract of
$Q_J$. If $\pi_J:Q\to Q_J$ is the coordinate projection, then
$e\pi_J|_X:X\to V$ is the identity on $V$. Thus $V$ is also a direct
summand of $X$, and both $V$ and $X/V$ belong to the specified class.
In the first case $|V|<\kappa$; since $|R|<\kappa$, a free presentation
of $V$ has fewer than $\kappa$ generators and relations.
In the second case $Q_J$ is a coproduct of fewer than $\kappa$ finitely
presented modules, hence is $\kappa$-presentable, as is its retract $V$.
\end{proof}

\begin{corol}\label{Shereditarypair}
For every set $\mathcal S$ of finitely presented modules,
\[
\bigl(G(\X_{\mathcal S}),G(\X_{\mathcal S})^{\perp *}\bigr)
\]
is a hereditary $*$-cotorsion pair in
$(R\text{-}\mathrm{Mod},\mathcal E_{\mathcal S})$.
\end{corol}

\begin{proof}
Apply \cref{cotorsion pair}, using \cref{Sbalancedpair,AddSKaplansky}
and local presentability of $R\text{-}\mathrm{Mod}$.
\end{proof}

\subsection{Periodic modules and small filtrations}

\begin{defin}\label{strongGSdefinition}
A module $M$ is $\mathcal S$-periodic if it admits an
$\mathcal E_{\mathcal S}$-conflation
\[
0\longrightarrow M\longrightarrow Q\longrightarrow M\longrightarrow0,
\qquad Q\in\X_{\mathcal S}.
\]
We denote this class by $\operatorname{Per}_{\mathcal S}(\X_{\mathcal S})$.
Such an $M$ is strongly Gorenstein $\X_{\mathcal S}$ if, in addition,
$M\in{}^{\perp *}\X_{\mathcal S}$.
\end{defin}

\begin{lem}\label{Speriodicnormalform}
Every $\mathcal S$-periodic module is a direct summand of an
$\mathcal S$-periodic module $N$ admitting a conflation
\begin{equation}\label{eq:Snormalform}
0\longrightarrow N\longrightarrow
\bigoplus_{\delta\in\Delta}S_\delta
\longrightarrow N\longrightarrow0,
\qquad S_\delta\in\mathcal S_0.
\end{equation}
Every strongly Gorenstein $\X_{\mathcal S}$-module belongs to
$G(\X_{\mathcal S})$, and every object of $G(\X_{\mathcal S})$ is a
direct summand of a strongly Gorenstein module with a conflation of
the form \eqref{eq:Snormalform}.
\end{lem}

\begin{proof}
For a periodic conflation with middle term $Q$, choose $Q'$ such that
$Q\oplus Q'=A=\bigoplus_{i\in I}S_i$, and put $B=A^{(\mathbb N)}$.
Reindexing coproducts gives $B\cong B\oplus B\cong Q\oplus B$.
Adding the split conflation $0\to B\to B\oplus B\to B\to0$ to the
given one produces a periodic conflation for $M\oplus B$ with middle
term $Q\oplus B\oplus B\cong B$.
If $M$ is strongly Gorenstein, so is $M\oplus B$, because
$B\in\X_{\mathcal S}$ is relative projective.

Splicing the conflation of a strongly Gorenstein module gives a right $\X_{\mathcal S}$-acyclic complex. Its left
$\X_{\mathcal S}$-acyclicity follows from the $\operatorname{Ext}^1_*$-vanishing of the strongly Gorenstein module.
Conversely, let $T$ be a right and left $\X_{\mathcal S}$-acyclic
complex with terms in $\X_{\mathcal S}$. Summing its cycle conflations
and reindexing gives
\[
0\longrightarrow\bigoplus_{n\in\mathbb Z}Z^n(T)
\longrightarrow\bigoplus_{n\in\mathbb Z}T^n
\longrightarrow\bigoplus_{n\in\mathbb Z}Z^n(T)
\longrightarrow0.
\]
This is an $\mathcal E_{\mathcal S}$-conflation, since every test is
finitely presented. For $X\in\X_{\mathcal S}$, each map
$Z^n(T)\to X$ extends to $T^n$ by left acyclicity. Taking products of
these surjections of Hom groups proves the required
$\operatorname{Ext}^1_*$-vanishing for $\bigoplus_n Z^n(T)$.
Every cycle is a direct summand of this module; the preceding argument supplies the asserted normal form.
\end{proof}

The next two lemmas adapt the periodic purification and filtration
constructions of \cite[Lemma~5.1 and Theorem~5.2]{CortesIzurdiaga2025}.
A small pure submodule of a periodic module need not, by itself,
give a periodic subobject of the chosen conflation. We therefore
construct the small submodule together with a subsum of
the middle term, preserving relative exactness for both the subobject
and the quotient.
For an infinite regular $\lambda$, a submodule is $\lambda$-pure if
every system of fewer than $\lambda$ linear equations in fewer than
$\lambda$ unknowns, with parameters in that submodule and solvable
in the ambient module, is solvable in the submodule.

\begin{lem}\label{relativepurificationS}
Let $\lambda$ be infinite regular, and let
\begin{equation}\label{eq:puritybound}
\mu=\mu^{<\lambda}\geq
\max\{|R|,|\mathcal S|,\lambda,\aleph_0\}.
\end{equation}
Suppose that
\[
\mathbb D:\quad
0\longrightarrow N\xrightarrow{i}Q=\bigoplus_{\delta\in\Delta}S_\delta
\xrightarrow{\pi}N\longrightarrow0,
\qquad S_\delta\in\mathcal S_0,
\]
is an $\mathcal E_{\mathcal S}$-conflation. For any $U\subseteq N$
with $|U|\leq\mu$, there are $E\subseteq\Delta$ and a submodule
$A\subseteq N$ such that $|E|,|A|\leq\mu$, $U\subseteq A$,
$A$ is $\lambda$-pure in $N$, and there is a commutative diagram
\[
\begin{tikzcd}[column sep=small]
0\arrow[r]&A\arrow[r,"i|_A"]\arrow[d,hook]&Q_E\arrow[r,"\pi|_{Q_E}"]\arrow[d,hook]&A\arrow[r]\arrow[d,hook]&0\\
0\arrow[r]&N\arrow[r,"i"]\arrow[d,two heads]&Q\arrow[r,"\pi"]\arrow[d,two heads]&N\arrow[r]\arrow[d,two heads]&0\\
0\arrow[r]&N/A\arrow[r,"\bar i"]&Q/Q_E\arrow[r,"\bar\pi"]&N/A\arrow[r]&0,
\end{tikzcd}
\qquad Q_E=\bigoplus_{\delta\in E}S_\delta.
\]
Its rows and its three nonzero columns are
$\mathcal E_{\mathcal S}$-conflations, with zeros understood at the
ends of the columns. The middle column splits.
\end{lem}

\begin{proof}
For $D\subseteq\Delta$, set $K_D=i^{-1}(Q_D)$.
If $|D|\leq\mu$, then $|Q_D|,|K_D|\leq\mu$, since the tests are
finitely generated and $|R|\leq\mu$. If $|C|\leq\mu$, there are at
most $\mu$ maps $S\to C$ with $S\in\mathcal S_0$.
We also use \cite[Lemma~4.1]{CortesIzurdiaga2025}: every submodule of
cardinality at most $\mu$ is contained in a $\lambda$-pure submodule
of cardinality at most $\mu$. Its stated bound by $\mu$ generators
is a cardinality bound here because $|R|\leq\mu$.

Construct increasing families $(B_\alpha)_{\alpha<\lambda}$ in $N$
and $(E_\alpha)_{\alpha<\lambda}$ in $\Delta$, all of cardinality
at most $\mu$. Start with $B_0=\langle U\rangle$, $E_0=\varnothing$.
At a successor step, choose a $\lambda$-pure submodule
$B_{\alpha+1}\subseteq N$ of cardinality at most $\mu$ containing
\[
B_\alpha+\pi(Q_{E_\alpha})+K_{E_\alpha}.
\]
Enlarge $E_\alpha$ to $E_{\alpha+1}$, still of cardinality at most
$\mu$, so that
\[
B_{\alpha+1}\subseteq\pi(Q_{E_{\alpha+1}}),\qquad
i(B_{\alpha+1})\subseteq Q_{E_{\alpha+1}},
\]
and every map $S\to B_{\alpha+1}$, $S\in\mathcal S_0$, has a chosen
lift to $Q_{E_{\alpha+1}}$ through $\pi$.
For the first two conditions, collect supports of chosen preimages
and of the elements of $i(B_{\alpha+1})$.
For the last, lift through $\pi$ and collect the supports of the
images of finite generating sets of the tests. The preceding counting
bounds justify all these choices. At limits take unions.

Put $A=\bigcup_{\alpha<\lambda}B_\alpha$ and
$E=\bigcup_{\alpha<\lambda}E_\alpha$. Then $|A|,|E|\leq\mu$.
Since elements of a direct sum have finite support,
$Q_E=\bigcup_{\alpha<\lambda}Q_{E_\alpha}$.
The inclusions $B_{\alpha+1}\subseteq\pi(Q_{E_{\alpha+1}})$
and $\pi(Q_{E_\alpha})\subseteq B_{\alpha+1}$ give
$A=\pi(Q_E)$. Also $i(B_{\alpha+1})\subseteq Q_{E_{\alpha+1}}$
implies $A\subseteq i^{-1}(Q_E)$. Conversely, if $i(x)\in Q_E$,
then $i(x)\in Q_{E_\alpha}$ for some $\alpha$, so
$x\in K_{E_\alpha}\subseteq B_{\alpha+1}\subseteq A$.
Thus
\[
A=\pi(Q_E)=i^{-1}(Q_E).
\]
Since the original row is exact, these identities prove exactness
of the upper row. Every map $S\to A$ has image in some
$B_{\alpha+1}$, by finite generation of $S$, and hence lifts to
$Q_E$. Thus the upper row is $\mathcal E_{\mathcal S}$-exact.
Regularity of $\lambda$ places the parameters of any system of fewer
than $\lambda$ equations in some $B_{\alpha+1}$. Its $\lambda$-purity
provides a solution there. Hence $A$ is $\lambda$-pure in $N$.

The quotient row is exact by the $3\times3$ lemma. A map
$S\to N/A$ lifts to $N$, since $A\subseteq N$ is pure and $S$ is
finitely presented, and then lifts to $Q$ by the middle row.
Passing to $Q/Q_E$ proves relative exactness of the quotient row.
Purity of $A\subseteq N$ and the splitting of the middle column
prove the assertion about the columns.
\end{proof}

\begin{lem}\label{Speriodicfiltration}
Under \eqref{eq:puritybound}, every conflation
\[
\mathbb E:\quad 0\to M\xrightarrow{i}Q\to M\to0
\]
of the form \eqref{eq:Snormalform} has a continuous filtration
by periodic subobjects
\[
\mathbb E_\alpha:\quad
0\longrightarrow M_\alpha\longrightarrow Q_\alpha
\longrightarrow M_\alpha\longrightarrow0
\qquad(\alpha\leq\tau)
\]
such that $M_0=Q_0=0$, $M_\tau=M$, and:
\begin{enumerate}[(i)]
\item $Q_\alpha$ is a coordinate subsum of the original middle term;
\item $\mathbb E_\alpha$ and $\mathbb E/\mathbb E_\alpha$ are
$\mathcal E_{\mathcal S}$-conflations, and $M_\alpha\to M$ is an
$\mathcal E_{\mathcal S}$-inflation;
\item $M_{\alpha+1}/M_\alpha$ is $\mathcal S$-periodic, has
cardinality at most $\mu$, and is $\lambda$-pure in $M/M_\alpha$.
\end{enumerate}
In particular, this is an $\mathcal E_{\mathcal S}$-filtration of $M$
by small periodic modules. 
\end{lem}

\begin{proof}
Enumerate $M$ as $(m_\alpha)_{\alpha<\tau}$, and start with
$M_0=Q_0=0$. At each stage we preserve the relative conflations
$\mathbb E_\alpha$ and $\mathbb E/\mathbb E_\alpha$, as well as
$0\to M_\alpha\to M\to M/M_\alpha\to0$. Given
$\mathbb E_\alpha$, apply \cref{relativepurificationS} to
$\mathbb E/\mathbb E_\alpha$ and the singleton
$\{m_\alpha+M_\alpha\}$. Write $A\subseteq M/M_\alpha$ for its
small end term and $Q'$ for its coordinate middle term. Define
$M_{\alpha+1}$ as the inverse image of $A$ in $M$, and
$Q_{\alpha+1}$ as the inverse image of $Q'$ in $Q$.
Then $Q_{\alpha+1}$ is a coordinate subsum, and the induced
periodic row is exact by the $3\times3$ lemma.
To check its relative exactness, map $S$ to $M_{\alpha+1}$,
project to $A$, and lift to $Q'$. Lift further through the split
epimorphism $Q_{\alpha+1}\to Q'$. The error maps to $M_\alpha$
and lifts to $Q_\alpha$ by induction.
The quotient periodic row is the quotient supplied by the
purification lemma.

The sequence $0\to M_\alpha\to M_{\alpha+1}\to A\to0$ is the
pullback of $0\to M_\alpha\to M\to M/M_\alpha\to0$.
Moreover, a map $S\to M/M_{\alpha+1}$ lifts first to
$M/M_\alpha$, by purity of $A$, and then to $M$ by induction.
Thus $M_\alpha\to M_{\alpha+1}$ and $M_{\alpha+1}\to M$ are relative inflations, and
$m_\alpha\in M_{\alpha+1}$.

At a limit ordinal $\beta$, take unions of the subobjects.
Their periodic rows and the quotient periodic rows remain
$\mathcal E_{\mathcal S}$-conflations by preservation of filtered
colimits, proved in \cref{Sbalancedpair}. Taking the colimit of
$0\to M_\alpha\to M\to M/M_\alpha\to0$ also proves that
$M_\beta\to M$ is a relative inflation. This completes the
recursion. Every element is included, so $M_\tau=M$; exactness
of the final quotient row also gives $Q_\tau=Q$.
\end{proof}

\subsection{Two completeness criteria}

The preceding lemmas give filtrations by small periodic modules.
To obtain filtrations suitable for $G(\X_{\mathcal S})$, we must
also control orthogonality to the relative projectives. We consider
two sufficient hypotheses:
\begin{align*}
(\mathrm{PI}_{\mathcal S})\quad &
\text{there is an infinite regular $\lambda$ such that every
$X\in\X_{\mathcal S}$ is $\lambda$-pure-injective};\\
(\mathrm{PO}_{\mathcal S})\quad &
\operatorname{Per}_{\mathcal S}(\X_{\mathcal S})
\subseteq{}^{\perp *}\X_{\mathcal S}.
\end{align*}
Here $\lambda$-pure-injectivity refers to ordinary module-theoretic
$\lambda$-purity: maps into the module extend across every
$\lambda$-pure inclusion.
Under $(\mathrm{PO}_{\mathcal S})$, every periodic module already
has the required orthogonality. Under $(\mathrm{PI}_{\mathcal S})$,
we preserve it along the filtration, using purity at successor stages
and the periodic conflations at limit stages.

\begin{theor}\label{Scompletenessmain}
If either $(\mathrm{PI}_{\mathcal S})$ or
$(\mathrm{PO}_{\mathcal S})$ holds, then
\[
\bigl(G(\X_{\mathcal S}),G(\X_{\mathcal S})^{\perp *}\bigr)
\]
is a functorially complete hereditary $*$-cotorsion pair.
More precisely, for a suitable infinite $\mu$, if $\mathcal C_\mu$
is a representative set of strongly Gorenstein $\X_{\mathcal S}$-modules
of cardinality at most $\mu$, then
\begin{equation}\label{eq:Smainfiltration}
\begin{gathered}
G(\X_{\mathcal S})=
\operatorname{Retr}\bigl(
\operatorname{Filt}_{\mathcal E_{\mathcal S}}(\mathcal C_\mu)\bigr),\\
G(\X_{\mathcal S})^{\perp *}=\mathcal C_\mu^{\perp *}.
\end{gathered}
\end{equation}
\end{theor}

\begin{proof}
Under $(\mathrm{PI}_{\mathcal S})$, use its cardinal $\lambda$;
under $(\mathrm{PO}_{\mathcal S})$, use $\lambda=\aleph_0$.
Choose $\mu$ satisfying \eqref{eq:puritybound}. Such cardinals
exist in ZFC: starting above the stated bounds, iterate
$\nu\mapsto\nu^{<\lambda}$ through a continuous chain of length
$\lambda$ and take its supremum. Regularity of $\lambda$ ensures
that the supremum is fixed by this operation.
By \cref{Speriodicnormalform}, it suffices to filter a strongly
Gorenstein module $M$ with middle term in normal form.
Choose the filtration supplied by \cref{Speriodicfiltration}.

Under $(\mathrm{PO}_{\mathcal S})$, the modules $M_\alpha$, $M/M_\alpha$, and
$M_{\alpha+1}/M_\alpha$ are all strongly Gorenstein, since each
is periodic. Under $(\mathrm{PI}_{\mathcal S})$, we prove the same
assertion by induction. It holds at $\alpha=0$.
At a successor step put $N=M/M_\alpha$ and
$A=M_{\alpha+1}/M_\alpha$. For $X\in\X_{\mathcal S}$, the sequence
$0\to A\to N\to N/A\to0$ is a relative conflation, and
\[
\operatorname{Hom}_R(N,X)\longrightarrow\operatorname{Hom}_R(A,X)
\longrightarrow\operatorname{Ext}^1_*(N/A,X)
\longrightarrow\operatorname{Ext}^1_*(N,X)
\]
is exact. The first map is surjective by $\lambda$-pure-injectivity,
and the last group is zero by induction. Thus $N/A$ is strongly
Gorenstein. Since $N,N/A\in G(\X_{\mathcal S})$,
\cref{Shereditarypair} gives $A\in G(\X_{\mathcal S})$.
In particular $A\in{}^{\perp *}\X_{\mathcal S}$, and its periodic
conflation makes it strongly Gorenstein. The relative extension
of $A$ by $M_\alpha$ has the same orthogonality, so
$M_{\alpha+1}$ is strongly Gorenstein as well.

At a limit ordinal $\beta$, the exact Eklof lemma
\cite[Proposition~5.7]{StovicekExactModels2014}, applied to the
previously constructed factors, gives
$\operatorname{Ext}^1_*(M_\beta,X)=0$ for every $X\in\X_{\mathcal S}$.
The periodic row makes $M_\beta$ strongly Gorenstein.
To prove the same orthogonality for $M/M_\beta$, let
$f:M_\beta\to X$. Applying $\operatorname{Hom}_R(-,X)$ to
\[
0\to M_\beta\xrightarrow{i_\beta}Q_\beta\to M_\beta\to0
\]
shows that $f$ extends to a map $h:Q_\beta\to X$.
Let $r_\beta:Q\to Q_\beta$ be the coordinate retraction, and let
$i:M\to Q$ be the original inclusion. Then $h r_\beta i:M\to X$
extends $f$, because $i|_{M_\beta}$ is the composite of $i_\beta$
with the inclusion $Q_\beta\to Q$. Hence
$\operatorname{Hom}_R(M,X)\to\operatorname{Hom}_R(M_\beta,X)$
is surjective. The relative long exact sequence for
$0\to M_\beta\to M\to M/M_\beta\to0$, together with
$\operatorname{Ext}^1_*(M,X)=0$, gives
$\operatorname{Ext}^1_*(M/M_\beta,X)=0$. Its periodic row makes
the quotient strongly Gorenstein. This argument does not require
$M_\beta\subseteq M$ to be $\lambda$-pure.

Thus every $G\in G(\X_{\mathcal S})$ is a retract of a
$\mathcal C_\mu$-filtered object. Conversely,
$\mathcal C_\mu\subseteq G(\X_{\mathcal S})
={}^{\perp *}(G(\X_{\mathcal S})^{\perp *})$ by
\cref{Shereditarypair}; the exact Eklof lemma and closure under
retracts give the first equality in \eqref{eq:Smainfiltration}.
Applying the same lemma to $\mathcal C_\mu^{\perp *}$ gives the
second. Since $|G_0|\leq\mu$ and relative projectives are strongly
Gorenstein, $\mathcal C_\mu$ contains a representative of $G_0$.
Efficiency from \cref{Sbalancedpair} and
\cite[Theorem~5.16]{StovicekExactModels2014} now yield functorial
completeness. Heredity is \cref{Shereditarypair}.
\end{proof}

\subsection{Consequences and bounded periodic tests}

\begin{corol}\label{SigmaPcomplete}
If $G_0$ is $\Sigma$-pure-injective, then the cotorsion pair in
\cref{Shereditarypair} is functorially complete.
In particular, this holds for $\X=\operatorname{Add}(P)$ and its
relative exact structure whenever $P$ is a finitely presented
$\Sigma$-pure-injective generator. It also holds for every finite
set $\mathcal S$ of finitely generated modules over an Artin algebra.
\end{corol}

\begin{proof}
Since $\X_{\mathcal S}=\operatorname{Add}(G_0)$, every member is
a direct summand of a coproduct of copies of $G_0$, hence is
pure-injective. Apply $(\mathrm{PI}_{\mathcal S})$ with
$\lambda=\aleph_0$.
For the single-generator assertion take $\mathcal S=\{P\}$;
$R\in\operatorname{Add}(P)$, since an epimorphism from a coproduct
of copies of $P$ to $R$ splits.
Finally, over an Artin algebra finite generation implies finite
presentation. If $\mathcal S$ is finite, then $G_0$ has finite
length over the commutative artinian coefficient ring. That ring
acts through $\operatorname{End}_R(G_0)$, so $G_0$ is endofinite.
Every endofinite module is $\Sigma$-pure-injective; see
\cite[Section~5]{Prest2008PureInjective}.
\end{proof}

\begin{corol}\label{Sfiniteidcomplete}
The following two conditions are equivalent:
\begin{enumerate}[(i)]
\item every $X\in\X_{\mathcal S}$ has finite
$\mathcal E_{\mathcal S}$-injective dimension;
\item $\sup\{\operatorname{id}_{\mathcal E_{\mathcal S}}X
\mid X\in\X_{\mathcal S}\}<\infty$.
\end{enumerate}
Either condition implies $(\mathrm{PO}_{\mathcal S})$ and hence
functorial completeness of the cotorsion pair in
\cref{Shereditarypair}.
\end{corol}

\begin{proof}
If (i) holds without a uniform bound, choose $X_n\in\X_{\mathcal S}$
with $\operatorname{id}_{\mathcal E_{\mathcal S}}X_n>n$.
Their coproduct belongs to $\X_{\mathcal S}$ and has injective
dimension at least that of each summand, contradicting (i).
The converse is immediate. For a periodic relative conflation
$0\to M\to Q\to M\to0$ with $Q\in\X_{\mathcal S}$, relative
projectivity of $Q$ gives
\[
\operatorname{Ext}^i_*(M,X)\cong
\operatorname{Ext}^{i+1}_*(M,X)\qquad(i\geq1).
\]
Finite relative injective dimension of $X$ forces these groups
to vanish. Thus $(\mathrm{PO}_{\mathcal S})$ holds.
\end{proof}

\begin{prop}\label{Sboundedperiodictests}
Fix an infinite $\mu\geq\max\{|R|,|\mathcal S|,\aleph_0\}$,
and let $\mathcal P_\mu$ be a representative set of the
$\mathcal S$-periodic modules of cardinality at most $\mu$. Then,
\[
\operatorname{Per}_{\mathcal S}(\X_{\mathcal S})
\subseteq\operatorname{Retr}\bigl(
\operatorname{Filt}_{\mathcal E_{\mathcal S}}(\mathcal P_\mu)\bigr),
\qquad
\operatorname{Per}_{\mathcal S}(\X_{\mathcal S})^{\perp *}
=\mathcal P_\mu^{\perp *}.
\]
Consequently, the $*$-cotorsion pair generated by
$\operatorname{Per}_{\mathcal S}(\X_{\mathcal S})$ is functorially
complete and hereditary. Moreover,
\[
(\mathrm{PO}_{\mathcal S})
\quad\Longleftrightarrow\quad
\mathcal P_\mu\subseteq{}^{\perp *}\X_{\mathcal S}.
\]
\end{prop}

\begin{proof}
Take $\lambda=\aleph_0$ in \cref{Speriodicfiltration} and use the
periodic normal form in \cref{Speriodicnormalform}.
This proves the first inclusion. Since
$\mathcal P_\mu\subseteq\operatorname{Per}_{\mathcal S}(\X_{\mathcal S})$,
the exact Eklof lemma and closure of left orthogonals under
retracts give equality of the right orthogonals.
Also $G_0\in\operatorname{Per}_{\mathcal S}(\X_{\mathcal S})$ and
$|G_0|\leq\mu$, so $\mathcal P_\mu$ contains a representative of
the exact-category generator. Apply
\cite[Theorem~5.16]{StovicekExactModels2014} and
\cref{Sbalancedpair} for functorial completeness.
For heredity, put
$\mathcal B=\operatorname{Per}_{\mathcal S}(\X_{\mathcal S})^{\perp *}$.
A periodic conflation gives, by dimension shifting,
$\operatorname{Ext}^n_*(M,B)\cong\operatorname{Ext}^1_*(M,B)=0$
for every periodic $M$, $B\in\mathcal B$, and $n\geq1$.
The long exact sequence in the second variable therefore shows that
$\mathcal B$ is coresolving. There are enough relative injectives by
\cref{Sbalancedpair}; dimension shifting along an injective
coresolution of $B$ now gives
$\operatorname{Ext}^n_*(L,B)=0$ for all $L\in{}^{\perp *}\mathcal B$
and $n\geq1$. Thus the cotorsion pair is hereditary.
Finally, if $\mathcal P_\mu\subseteq{}^{\perp *}\X_{\mathcal S}$,
the first inclusion and the exact Eklof lemma imply that every
periodic module lies in ${}^{\perp *}\X_{\mathcal S}$.
The reverse implication is immediate.
\end{proof}

\begin{remark}\label{Spoacyclic}
Condition $(\mathrm{PO}_{\mathcal S})$ is also equivalent to saying
that every right $\X_{\mathcal S}$-acyclic complex with terms in
$\X_{\mathcal S}$ is left $\X_{\mathcal S}$-acyclic.
Indeed, sum its cycle conflations to obtain a periodic module.
Under $(\mathrm{PO}_{\mathcal S})$, each cycle belongs to
${}^{\perp *}\X_{\mathcal S}$ as a direct summand of that module;
applying $\operatorname{Hom}_R(-,X)$ to the cycle conflations proves
left acyclicity. Conversely, splice any periodic conflation and
apply the asserted left acyclicity to obtain its orthogonality.
\end{remark}

\subsection{A relative cotorsion pair over a Dedekind domain}
\label{subsec:Dedekind}

We now compute the relative Gorenstein objects for a family of finitely
presented tests over a Dedekind domain. The computation uses the hereditary property of Dedekind domains
and the theory of quasi-Frobenius rings; see \cite[Sections~2C, 2E and Chapter~6]{Lam1999ModulesRings}.
In particular, the fact that every module over a quasi-Frobenius ring
is Gorenstein projective is standard; see also
\cite[Theorem~2.2]{BennisMahdouOuarghi2010}.
We identify the resulting classes and approximations in the exact structure $\mathcal E_{\mathcal S}$ on $R\text{-}\mathrm{Mod}$.

Throughout this subsection, $R$ is a commutative Dedekind domain,
$\Lambda$ is a set of nonzero prime ideals, and
$n_{\mathfrak p}\geq1$ is an integer for each $\mathfrak p\in\Lambda$.
Set
\[
A_{\mathfrak p}=R/\mathfrak p^{\,n_{\mathfrak p}},\qquad
\mathcal S=\{A_{\mathfrak p}:\mathfrak p\in\Lambda\},\qquad
\X=\X_{\mathcal S}.
\]
For an ideal $I$ and an $R$-module $N$, write
$N[I]=\{x\in N:Ix=0\}$.
The set $\Lambda$ may be infinite, and all module classes below consist
of arbitrary modules. The balanced pair and relative extension groups
are those of \cref{Sbalancedpair}.

\begin{lem}\label{Dedekindstructure}
Every submodule of an object of $\X$ is isomorphic to
\[
P\oplus\bigoplus_{\mathfrak p\in\Lambda}T_{\mathfrak p},
\qquad P\in\operatorname{Proj}R,\quad
T_{\mathfrak p}\in A_{\mathfrak p}\text{-}\mathrm{Mod}.
\]
Moreover, $X\in\X$ if and only if it has such a decomposition with
each $T_{\mathfrak p}$ projective over $A_{\mathfrak p}$.
Each $A_{\mathfrak p}$ is quasi-Frobenius, so its projective and
injective modules coincide.
\end{lem}

\begin{proof}
An object of $\X$ is a retract of a module
\[
D=R^{(I)}\oplus\bigoplus_{\mathfrak p\in\Lambda}
A_{\mathfrak p}^{(I_{\mathfrak p})}.
\]
Let $M\subseteq D$. Its ordinary torsion submodule $t(M)$ is
$M\cap t(D)$. The Chinese remainder theorem, applied to the finitely
many components supporting any one element, shows that
\[
t(M)=\bigoplus_{\mathfrak p\in\Lambda}T_{\mathfrak p},
\qquad T_{\mathfrak p}=M\cap A_{\mathfrak p}^{(I_{\mathfrak p})}.
\]
Indeed, each primary component of an element of $t(M)$ is a scalar
multiple of that element and therefore belongs to $M$.
Projection to $R^{(I)}$ identifies $M/t(M)$ with a submodule of a
free module. Since $R$ is hereditary, this quotient is projective;
thus $0\to t(M)\to M\to M/t(M)\to0$ splits.

If $M=X$ is a retract of $D$, a retraction $D\to X$ restricts on
each primary component to a retraction
$A_{\mathfrak p}^{(I_{\mathfrak p})}\to T_{\mathfrak p}$.
Hence each $T_{\mathfrak p}$ is $A_{\mathfrak p}$-projective.
The converse follows by taking direct sums of retracts of free
$R$-modules and free $A_{\mathfrak p}$-modules.

Finally,
$A_{\mathfrak p}\cong R_{\mathfrak p}/(\pi^{n_{\mathfrak p}})$,
where $R_{\mathfrak p}$ is a discrete valuation ring with uniformizer
$\pi$. This is an artinian principal ideal local ring. It is
self-injective by Baer's criterion: a map from the ideal $(\pi^j)$
to $A_{\mathfrak p}$ sends $\pi^j$ to an element annihilated by
$\pi^{n_{\mathfrak p}-j}$, hence to $\pi^j b$ for some
$b\in A_{\mathfrak p}$, and therefore extends by multiplication by
$b$. Thus $A_{\mathfrak p}$ is quasi-Frobenius. The coincidence of
projective and injective modules, including infinitely generated ones,
is the standard quasi-Frobenius characterization
\cite[Chapter~6]{Lam1999ModulesRings}.
\end{proof}

\begin{theor}\label{Dedekindpair}
In the above setup, one has
\begin{align}
G(\X)
&=\left\{P\oplus\bigoplus_{\mathfrak p\in\Lambda}T_{\mathfrak p}:
P\in\operatorname{Proj}R,\quad
T_{\mathfrak p}\in A_{\mathfrak p}\text{-}\mathrm{Mod}\right\},
\label{eq:DedekindG}\\
G(\X)^{\perp *}
&=\left\{N:
N[\mathfrak p^{\,n_{\mathfrak p}}]\in
\operatorname{Inj}(A_{\mathfrak p}\text{-}\mathrm{Mod})
\text{ for all }\mathfrak p\in\Lambda\right\}.
\label{eq:DedekindB}
\end{align}
More precisely, for $T\in A_{\mathfrak p}\text{-}\mathrm{Mod}$ and
$N\in R\text{-}\mathrm{Mod}$ there are natural isomorphisms
\begin{equation}\label{eq:DedekindExt}
\operatorname{Ext}^i_*(T,N)
\cong\operatorname{Ext}^i_{A_{\mathfrak p}}
\bigl(T,N[\mathfrak p^{\,n_{\mathfrak p}}]\bigr),\qquad i\geq1.
\end{equation}
Condition $(\mathrm{PO}_{\mathcal S})$ holds, and
$\bigl(G(\X),G(\X)^{\perp *}\bigr)$ is a functorially complete
hereditary $*$-cotorsion pair.
\end{theor}

\begin{proof}
Every object of $G(\X)$ embeds in an object of $\X$, so
\cref{Dedekindstructure} gives one inclusion in
\eqref{eq:DedekindG}. For the converse, fix $\mathfrak p\in\Lambda$
and $T\in A_{\mathfrak p}\text{-}\mathrm{Mod}$.
Splice a projective resolution of $T$ over $A_{\mathfrak p}$ with
an injective coresolution. By \cref{Dedekindstructure}, the latter
also has projective terms. The resulting exact complex
$C_{\mathfrak p}$ has $T$ as a cycle and has all its terms in $\X$.

For every $A_{\mathfrak p}$-module $V$,
\[
\operatorname{Hom}_R(A_{\mathfrak q},V)=0
\quad(\mathfrak q\neq\mathfrak p),\qquad
\operatorname{Hom}_R(A_{\mathfrak p},V)\cong V.
\]
The first equality follows from comaximality of the relevant ideal
powers. It follows that $C_{\mathfrak p}$ is right $\X$-acyclic:
exactness can be tested on $R$ and the $A_{\mathfrak q}$, and then
passes to coproducts and retracts of these tests by taking products
and retracts of the corresponding Hom complexes.

Let $X\in\X$. By \cref{Dedekindstructure}, write
$X=P\oplus\bigoplus_{\mathfrak q}Q_{\mathfrak q}$, where $P$ is
$R$-projective and $Q_{\mathfrak q}$ is $A_{\mathfrak q}$-projective.
Any map from an $A_{\mathfrak p}$-module to $X$ has image in
$Q_{\mathfrak p}$: $P$ is torsion-free and the other primary
components have no elements annihilated by
$\mathfrak p^{\,n_{\mathfrak p}}$. Hence
\[
\operatorname{Hom}_R(C_{\mathfrak p},X)
\cong\operatorname{Hom}_{A_{\mathfrak p}}
(C_{\mathfrak p},Q_{\mathfrak p})
\]
is exact, since $Q_{\mathfrak p}$ is injective over
$A_{\mathfrak p}$. Thus $T\in G(\X)$.

For a family $(T_{\mathfrak p})$, take the direct sum of these
complete resolutions. It is right $\X$-acyclic because the tests
are finitely presented and direct sums of exact module complexes
are exact. It is left $\X$-acyclic because applying
$\operatorname{Hom}_R(-,X)$ gives a product of exact complexes of
abelian groups, and products of abelian groups are exact.
Ordinary projectives already belong to $\X\subseteq G(\X)$.
This proves \eqref{eq:DedekindG}.

An ordinary projective resolution of $T$ over $A_{\mathfrak p}$ is,
by the same test calculation, an $\mathcal E_{\mathcal S}$-projective
resolution over $R$. On its terms the natural identification
\[
\operatorname{Hom}_R(V,N)
=\operatorname{Hom}_{A_{\mathfrak p}}
\bigl(V,N[\mathfrak p^{\,n_{\mathfrak p}}]\bigr)
\]
proves \eqref{eq:DedekindExt}. Taking direct sums of these relative
resolutions and applying $\operatorname{Hom}_R(-,N)$ also gives
\[
\operatorname{Ext}^1_*
\left(P\oplus\bigoplus_{\mathfrak p}T_{\mathfrak p},N\right)
\cong\prod_{\mathfrak p}
\operatorname{Ext}^1_{A_{\mathfrak p}}
\bigl(T_{\mathfrak p},N[\mathfrak p^{\,n_{\mathfrak p}}]\bigr).
\]
Here exactness of products justifies passage to cohomology.
Vanishing for all $A_{\mathfrak p}$-modules $T_{\mathfrak p}$ is
equivalent to injectivity of the indicated target, proving
\eqref{eq:DedekindB}.

Every $\mathcal S$-periodic module is a submodule of an object of
$\X$, hence belongs to $G(\X)$ by \cref{Dedekindstructure} and
\eqref{eq:DedekindG}. Thus it is left orthogonal to $\X$ and
$(\mathrm{PO}_{\mathcal S})$ holds. The hereditary cotorsion-pair
conclusion follows from \cref{Shereditarypair}, which invokes
Section~3 using the Kaplansky lemma. Functorial completeness follows
separately from \cref{Scompletenessmain}.
\end{proof}

\begin{prop}\label{Dedekindapproximations}
Let $\mathcal B=G(\X)^{\perp *}$ in the setup of
\cref{Dedekindpair}. Every $R$-module $M$ admits an
$\mathcal E_{\mathcal S}$-conflation
\begin{equation}\label{eq:Dedekindprecover}
0\longrightarrow P_1\longrightarrow
P_0\oplus T(M)\longrightarrow M\longrightarrow0,
\qquad
T(M)=\bigoplus_{\mathfrak p\in\Lambda}
M[\mathfrak p^{\,n_{\mathfrak p}}],
\end{equation}
where $P_0,P_1$ are ordinary projective $R$-modules. Its last map
is a special $G(\X)$-precover. There is also an
$\mathcal E_{\mathcal S}$-conflation
$0\to M\to B\to H\to0$ with $B\in\mathcal B$ and $H\in G(\X)$,
so its first map is a special $\mathcal B$-preenvelope.
In particular, every module has a proper $G(\X)$-resolution of
length at most one.
\end{prop}

\begin{proof}
Comaximality makes the natural map $T(M)\to M$ injective.
Choose an ordinary projective presentation
$0\to P_1\to P_0\to M/T(M)\to0$.
Its kernel $P_1$ is projective because $R$ is hereditary.
Form the pullback $E=M\times_{M/T(M)}P_0$. The exact sequence
$0\to T(M)\to E\to P_0\to0$ splits, giving
$E\cong P_0\oplus T(M)$. The other pullback sequence is
$0\to P_1\to E\to M\to0$.
Every map $A_{\mathfrak p}\to M$ has image in $T(M)$ and lifts
through $T(M)\to E$; the test $R$ lifts because $E\to M$ is epic.
Hence this sequence is an $\mathcal E_{\mathcal S}$-conflation.
By \cref{Dedekindpair}, $E\in G(\X)$ and $P_1\in\mathcal B$,
the latter because $P_1$ is torsion-free. Applying
$\operatorname{Hom}_R(G,-)$ for $G\in G(\X)$ and using
$\operatorname{Ext}^1_*(G,P_1)=0$ proves that $E\to M$ is a
special precover and that \eqref{eq:Dedekindprecover} is proper.

For the other approximation, choose an
$\mathcal E_{\mathcal S}$-conflation $0\to M\to Y\to C\to0$
with $Y$ relative injective, using \cref{Sbalancedpair}.
Apply the preceding construction to $C$, obtaining
$0\to P\to H\to C\to0$ with $P$ ordinary projective and
$H\in G(\X)$. The pullback $B=Y\times_C H$ gives conflations
\[
0\longrightarrow M\longrightarrow B\longrightarrow H\longrightarrow0,
\qquad
0\longrightarrow P\longrightarrow B\longrightarrow Y\longrightarrow0.
\]
Both $P$ and $Y$ belong to $\mathcal B$, and right Ext-orthogonals
are closed under relative extensions, so $B\in\mathcal B$.
The first sequence is the required special preenvelope, since
$\operatorname{Ext}^1_*(H,B')=0$ for every $B'\in\mathcal B$.
Functorial completeness was established in \cref{Dedekindpair};
the constructions above give explicit approximation conflations.
\end{proof}

\begin{ex}\label{Dedekindintegers}
Let $R=\mathbb Z$ and
$\mathcal S=\{\mathbb Z/p^2:p\text{ prime}\}$.
Then
\[
G(\X)=\left\{F\oplus\bigoplus_p T_p:
F\text{ free abelian},\quad p^2T_p=0\right\},
\]
and its right relative orthogonal consists of the groups $N$ for
which $N[p^2]$ is injective, equivalently projective, over
$\mathbb Z/p^2$ for every prime $p$.
In particular, $\mathbb Z/p\in G(\X)\setminus\X$: it is not
projective over $\mathbb Z/p^2$.
Also $\mathbb Q\notin G(\X)$, since it is torsion-free and not
free. Thus $\X\subsetneq G(\X)\subsetneq\mathbb Z\text{-}\mathrm{Mod}$.
The tests cannot all belong to the additive closure of a single
finitely presented abelian group: such a group has nonzero torsion
at only finitely many primes. The quasi-Frobenius behavior of the
individual quotients is standard; compare
\cite[Corollaries~3.9--3.10]{BennisMahdouOuarghi2010}.
\end{ex}

\noindent\textbf{Funding.} K.~Golfis was supported by the Hellenic Foundation for Research and Innovation (H.F.R.I.) under the ``3rd Call for H.F.R.I. Research Projects to Support Faculty Members and Researchers'', project number 24921.

\bigskip
\noindent\begin{minipage}{\textwidth}
\small
Konstantinos Golfis,\\
Department of Mathematics, University of Athens, Athens 15784, Greece\\
E-mail: \href{mailto:golfisk@math.uoa.gr}{\texttt{golfisk@math.uoa.gr}}
\end{minipage}
\end{document}